%% file: 07-01.tex
\documentclass[reqno,final]{amsart}
\usepackage{natbib}  
\usepackage{fancyhdr} 
\usepackage{color} 
\usepackage{hyperref} 
\usepackage{graphicx} 

\definecolor{aleacolor}{rgb}{0.16,0.59,0.78}

\hypersetup{
breaklinks,
colorlinks=true,
linkcolor=aleacolor,
urlcolor=aleacolor,
citecolor=aleacolor}

\renewcommand{\headheight}{11pt}
\renewcommand{\cite}{\citet}

\theoremstyle{plain}
\newtheorem{theorem}{Theorem}[section]                                          
                          
\newtheorem{lemma}[theorem]{Lemma}
\newtheorem{corollary}[theorem]{Corollary}
\newtheorem{conjecture}[theorem]{Conjecture}
\theoremstyle{definition}
\newtheorem{definition}[theorem]{Definition}
\theoremstyle{remark}

\makeatletter \@addtoreset{equation}{section} \makeatother
\renewcommand\theequation{\thesection.\arabic{equation}}
\renewcommand\thefigure{\thesection.\arabic{figure}}
\renewcommand\thetable{\thesection.\arabic{table}}

\newcommand{\aleaIndex}[1]{\href{http://alea.impa.br/english/index_v#1.htm}{\bf #1}}
\eheader{Alea}{\aleaIndex{7}}{2010}{1}{18}

\elogo{\framebox[6cm][c]{\footnotesize \parbox[c]{5.5cm}{
The original article is published by the \href{http://alea.impa.br/english/index_v7.htm}{Latin American Journal of Probability and Mathematical Statistics}
} } \parbox[c]{3cm}{\includegraphics[width=3cm]{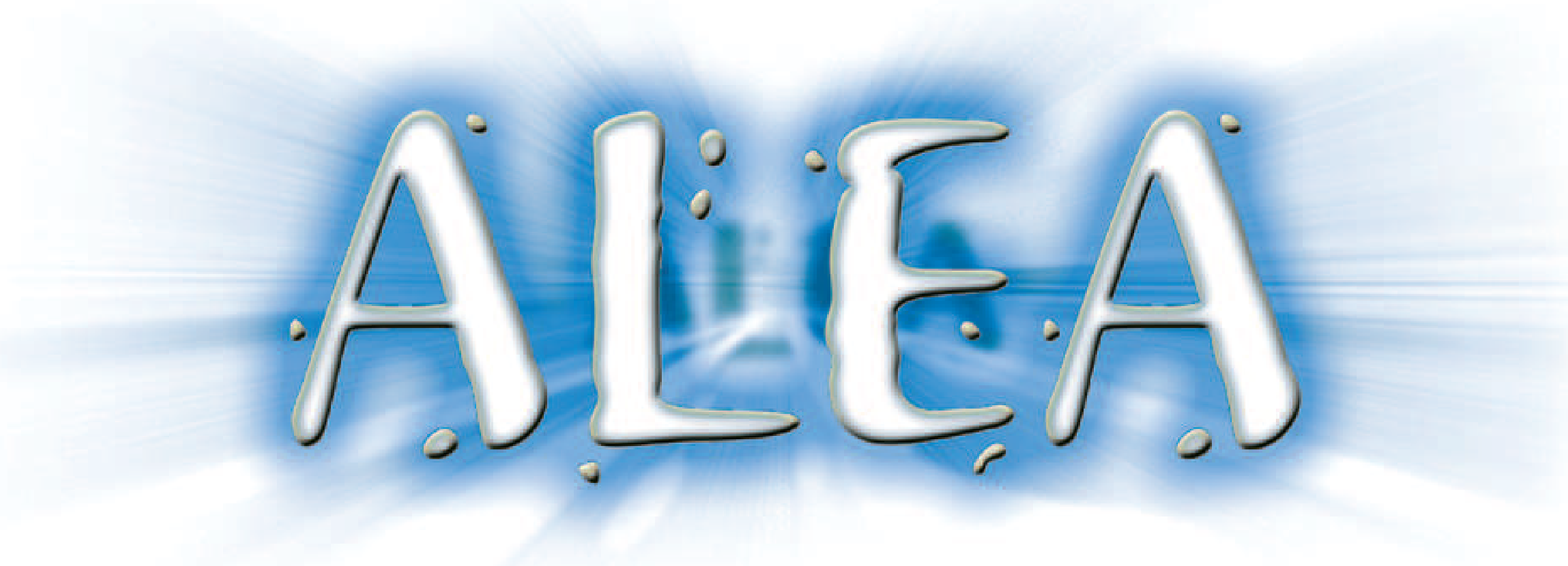}}}

\usepackage{epsfig}
\usepackage{subfigure}
\usepackage{amssymb}

\renewenvironment{proof}{\noindent{\textit{Proof:}}}{\hspace{2mm} $\square$}
\newenvironment{demon}[1]{\noindent{\textit{Proof of Theorem \ref{#1}:}}}{\hspace{2mm} $\square$}

\newcommand{\Z}{\mathbb{Z}}
\newcommand{\ind}{\hbox{{\small 1} \hspace*{-11pt} 1}}
\newcommand{\ep}{\epsilon}

\newcommand{\integer}[1]{\lceil #1 \rceil}

\DeclareMathOperator{\card}{card}

\DeclareMathOperator{\bin}{Binomial}

\begin{document}

\date{March 10, 2009; accepted March 8, 2010}

\author{N. Lanchier}
\address{Arizona State University\\\vskip .05cm
\noindent School of Mathematical and Statistical Sciences, Tempe, AZ 85287, USA.}
\email{lanchier@math.asu.edu}
\urladdr{\url{http://math.la.asu.edu/~lanchier}}
\title[Opinion dynamics with confidence threshold]{Opinion dynamics with confidence threshold: \\ an alternative to the Axelrod model} 

\begin{abstract}
 The voter model and the Axelrod model are two of the main stochastic processes that describe the spread of opinions on networks.
 The former includes social influence, the tendency of individuals to become more similar when they interact, while the latter
 also accounts for homophily, the tendency to interact more frequently with individuals which are more similar.
 The Axelrod model has been extensively studied during the past ten years based on numerical simulations.
 In contrast, we give rigorous analytical results for a generalization of the voter model that is closely related to the Axelrod
 model as it combines social influence and confidence threshold, which is modeled somewhat similarly to homophily.
 Each vertex of the network, represented by a finite connected graph, is characterized by an opinion and may interact with its
 adjacent vertices.
 Like the voter model, an interaction results in an agreement between both interacting vertices -- social influence -- but unlike
 the voter model, an interaction takes place if and only if the vertices' opinions are within a certain distance -- confidence threshold.
 In a deterministic static approach, we first give lower and upper bounds for the maximum number of opinions that can be
 supported by the network as a function of the confidence threshold and various characteristics of the graph.
 The number of opinions coexisting at equilibrium is then investigated in a probabilistic dynamic approach for the stochastic
 process starting from a random configuration.
 It is proved that, for large confidence thresholds, the probability of an ultimate consensus on any connected graph is
 strictly positive whereas, with high probability, any fraction of the opinions is retained by the dynamics in one dimension
 provided the confidence threshold is small enough.
\end{abstract}

\maketitle


\section{Introduction}
\label{sec:intro}

\indent Opinion and cultural dynamics are driven by social influence, which is the tendency of individuals to become
 more similar when they interact.
 At least when the number of interacting agents is finite and the network of interactions is connected, models including this
 aspect, such as the voter model introduced independently by \cite{clifford_sudbury_1973} and \cite{holley_liggett_1975},
 usually predict convergence to a global consensus.
 However, differences between individuals and groups persist in the real world.
 In his seminal paper \cite{axelrod_1997}, political scientist Robert Axelrod explains the diversity of opinions and cultures
 as a consequence of homophily, which is the tendency to interact more frequently with individuals which are more similar.
 In the Axelrod model, actors are characterized by a finite number of cultural features.
 In Axelrod's own words, the more similar an actor is to a neighbor, the more likely that actor will adopt one of the
 neighbor's traits.
 Interactions take place on a finite connected graph $G$ with vertex set $V$ and edge set $E$.
 Each vertex $x$ is characterized by a vector of $n$ cultural features, each of which having $m$ possible states,
 $$ c (x) \ = \ (f_1 (x), \ldots, f_n (x)) \ \ \hbox{where} \ \
    f_j (x) \in \{1, 2, \ldots, m \} \ \hbox{for} \ j = 1, 2, \ldots, n. $$
 At each update, a vertex, say $x$, is picked uniformly at random along with one of its neighbors, say $y$.
 Then, with a probability equal to the fraction of features vertices $x$ and $y$ have in common, one of the features for which
 states are different (if any) is selected uniformly at random, and the state of vertex $y$ is set equal to the state of vertex $x$
 for this cultural feature.
 Otherwise nothing happens.
 The system evolving in continuous-time with each ordered pair of neighbors becoming active at rate 1 can be modeled
 formally by the continuous-time Markov process whose state at time $t$ is a function $c_t$ that maps the vertex set of the graph
 into the set of cultures $\{1, 2, \ldots, m \}^n$, and whose dynamics are described by the Markov generator $L_A$ defined on
 the set of cylinder functions by
 $$ L_A \,g (c) \ = \ \sum_{x \in V} \ \sum_{y \sim x} \ \sum_{i = 1}^n \ \frac{1}{n} \ \bigg[\frac{f (x, y)}{1 - f (x, y)} \bigg] \
        \ind \{f_i (x) \neq f_i (y) \} \ [g (c_{x \to y}^i) - g (c)] $$
 where $y \sim x$ means that $x$ and $y$ are connected by an edge,
 $$ c_{x \to y}^i (y) \ = \ (f_1 (y), \ldots, f_{i - 1} (y), f_i (x), f_{i + 1} (y), \ldots, f_n (y)) $$
 and $c_{x \to y}^i (z) = c (z)$ for all $z \neq y$, and
 $$ f (x, y) \ = \ \frac{1}{n} \ \sum_{j = 1}^n \ \ind \{f_j (x) = f_j (y)\} $$
 is the fraction of traits vertices $x$ and $y$ share.
 Due to the finiteness of the graph, the system is always driven to an absorbing state:
 either an ordered configuration in which all the vertices share the same culture, or a disordered configuration in which
 different cultures coexist.
 The phase transition between these two regimes, which depends qualitatively and quantitatively upon the structure of the
 graph, the number of cultural features, and the number of states per feature, has been extensively studied by social
 scientists as well as statistical physicists based either on numerical simulations (see \citealp{gonzalez_all_2005, klemm_all_2003})
 or simple mean field treatments (see \citealp{castellano_all_2000, vazquez_redner_2007, vilone_all_2002}) ignoring the structure
 of the network, and we refer the reader to Section IV.A of \cite{castellano_fortunato_loreto_2009} for a review.

\begin{figure}[t]
\centering
\mbox{\subfigure[$\ep = 0$ (160,000)]  {\epsfig{figure = 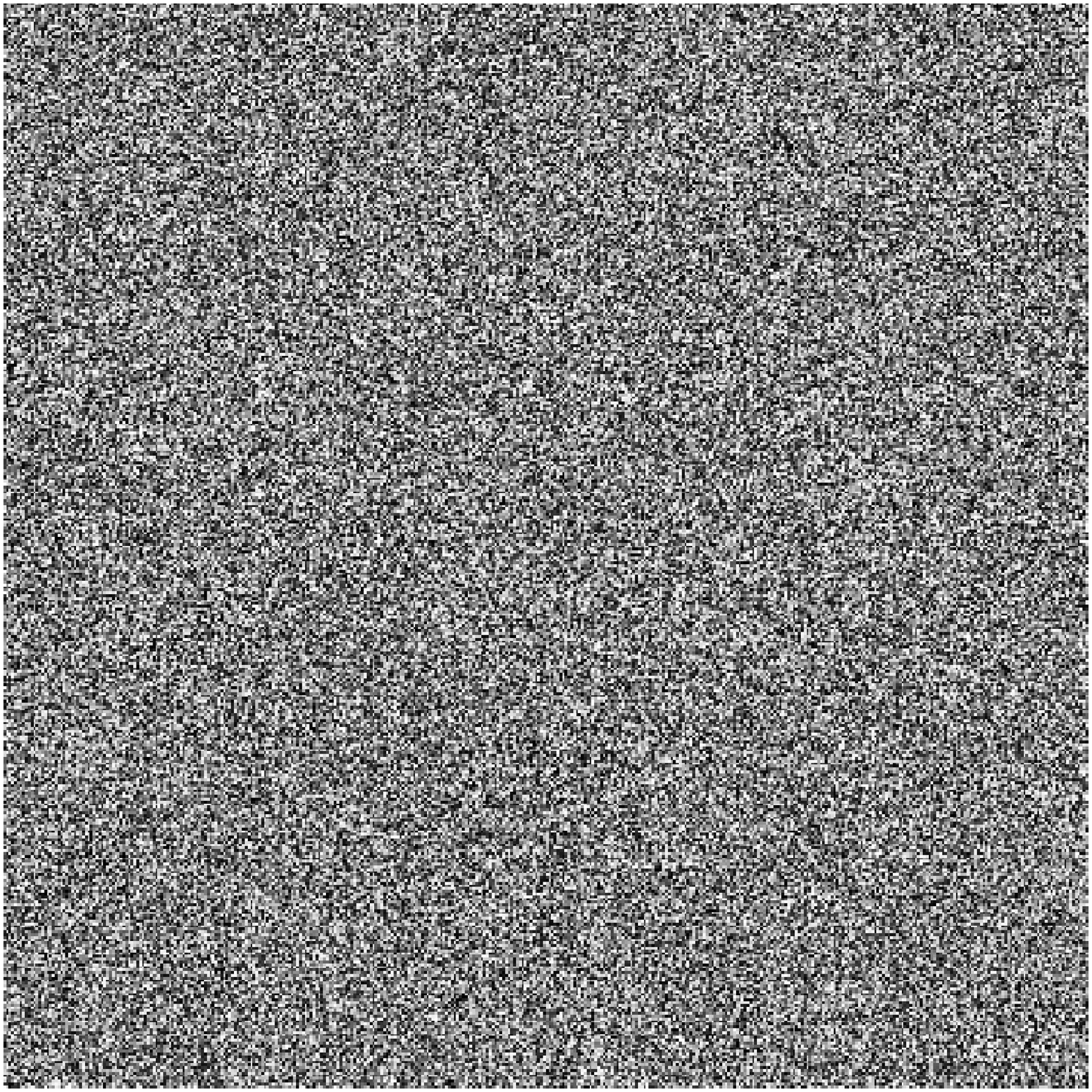, width = 150pt}} \hspace{5pt}
      \subfigure[$\ep = 1/5$ (22,126)] {\epsfig{figure = 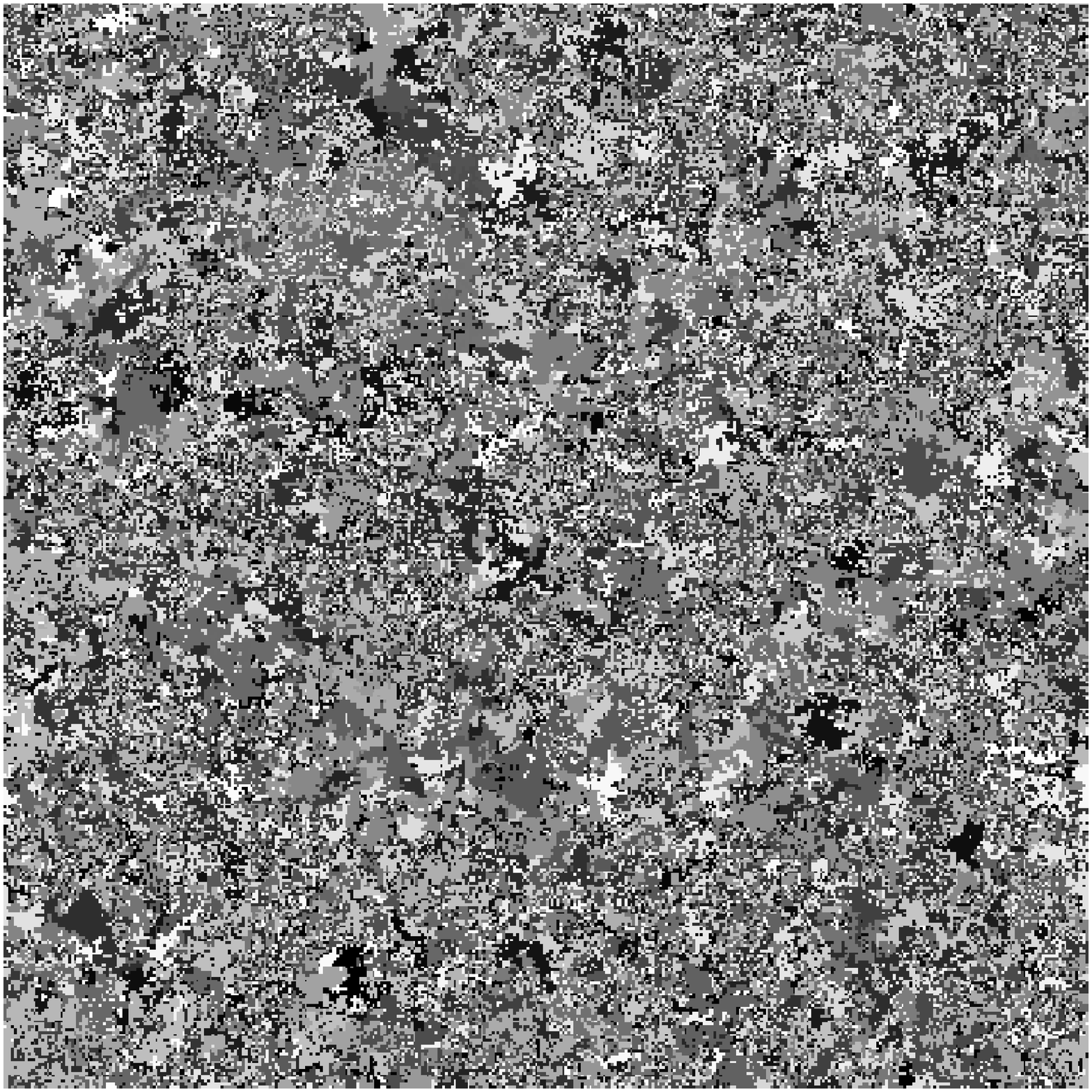, width = 150pt}}} \\ \vspace{-5pt}
\mbox{\subfigure[$\ep = 1/4$ (3,365)]  {\epsfig{figure = 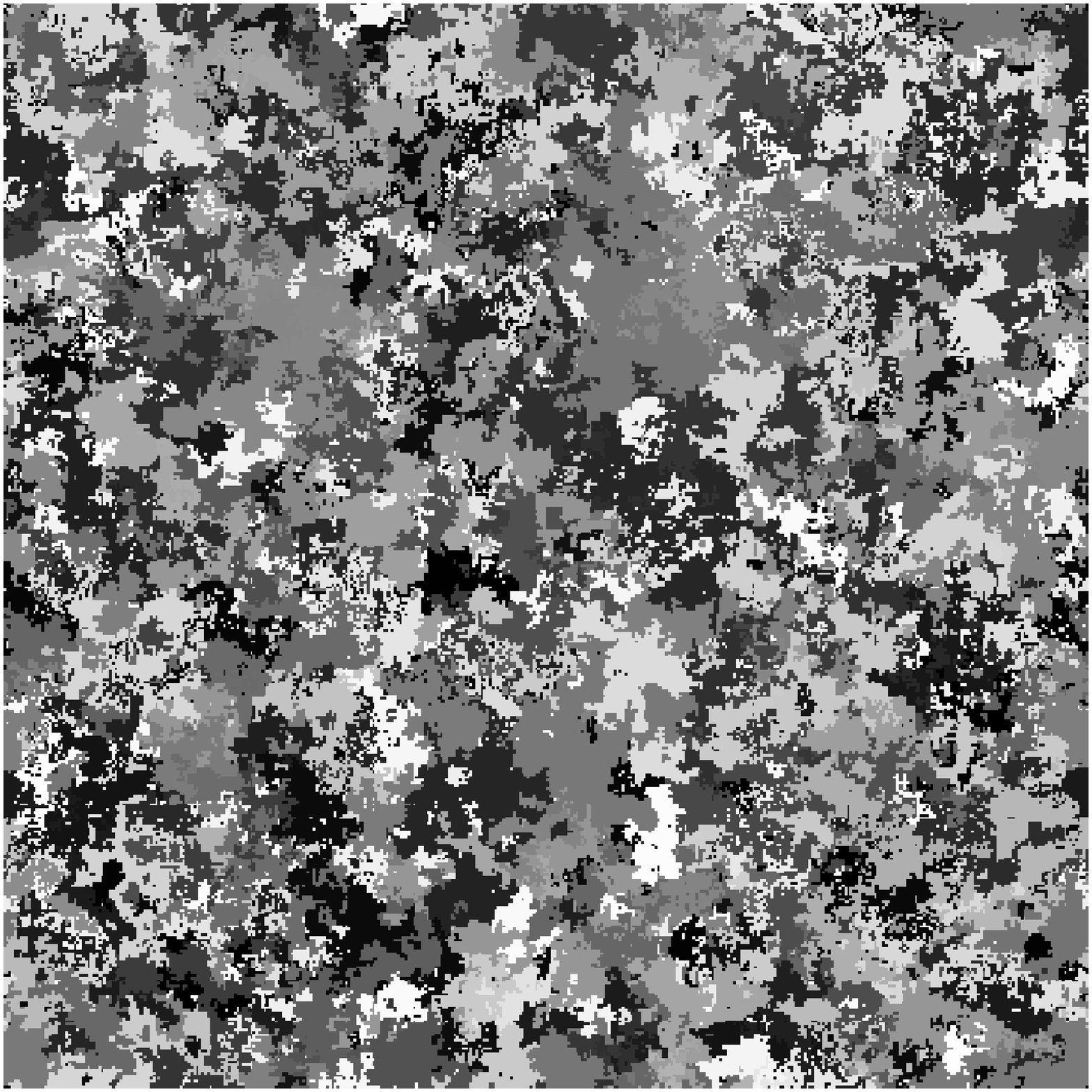, width = 150pt}} \hspace{5pt}
      \subfigure[$\ep = 1/3$ (594)]    {\epsfig{figure = 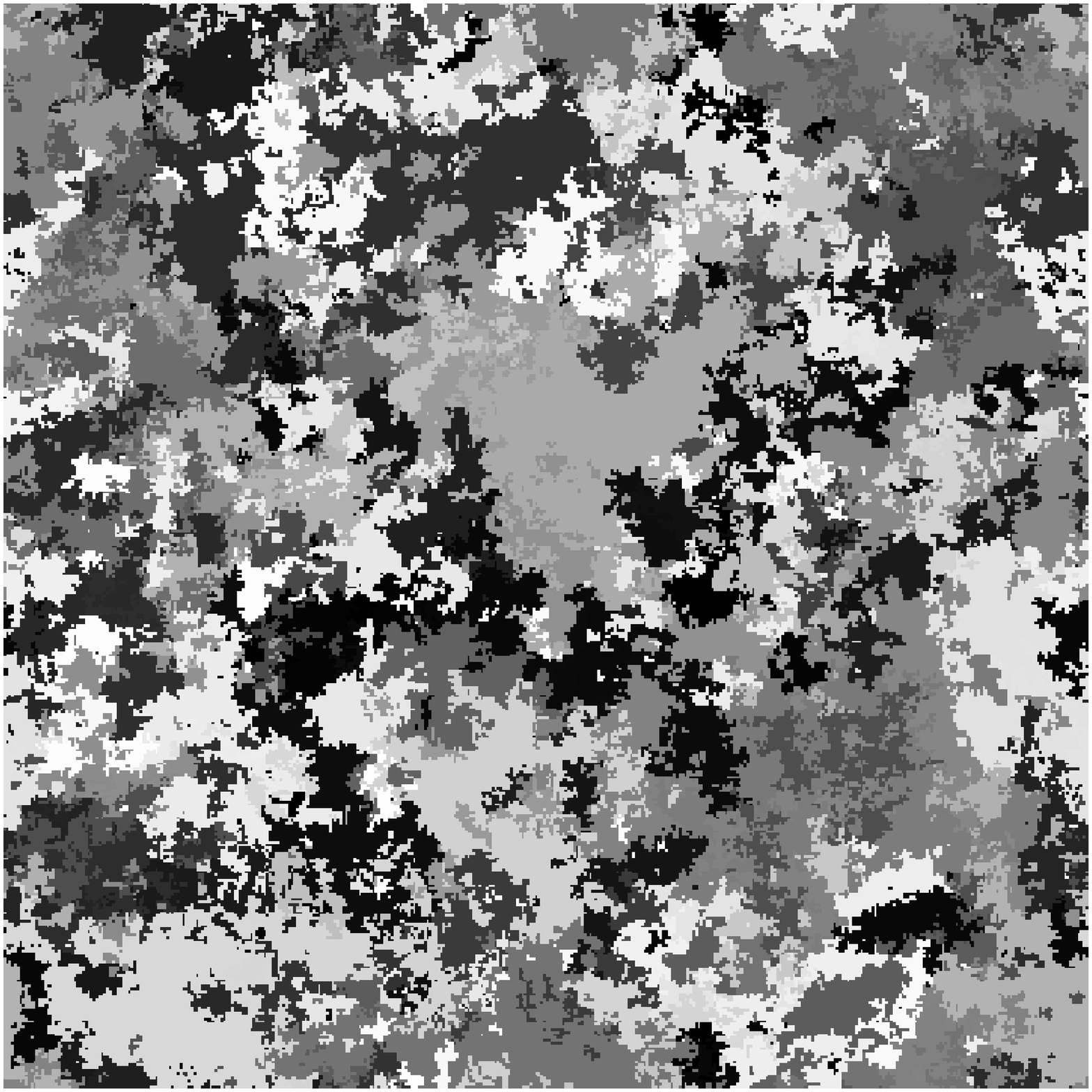, width = 150pt}}} \\ \vspace{-5pt}
\mbox{\subfigure[$\ep = 1/2$ (88)]     {\epsfig{figure = 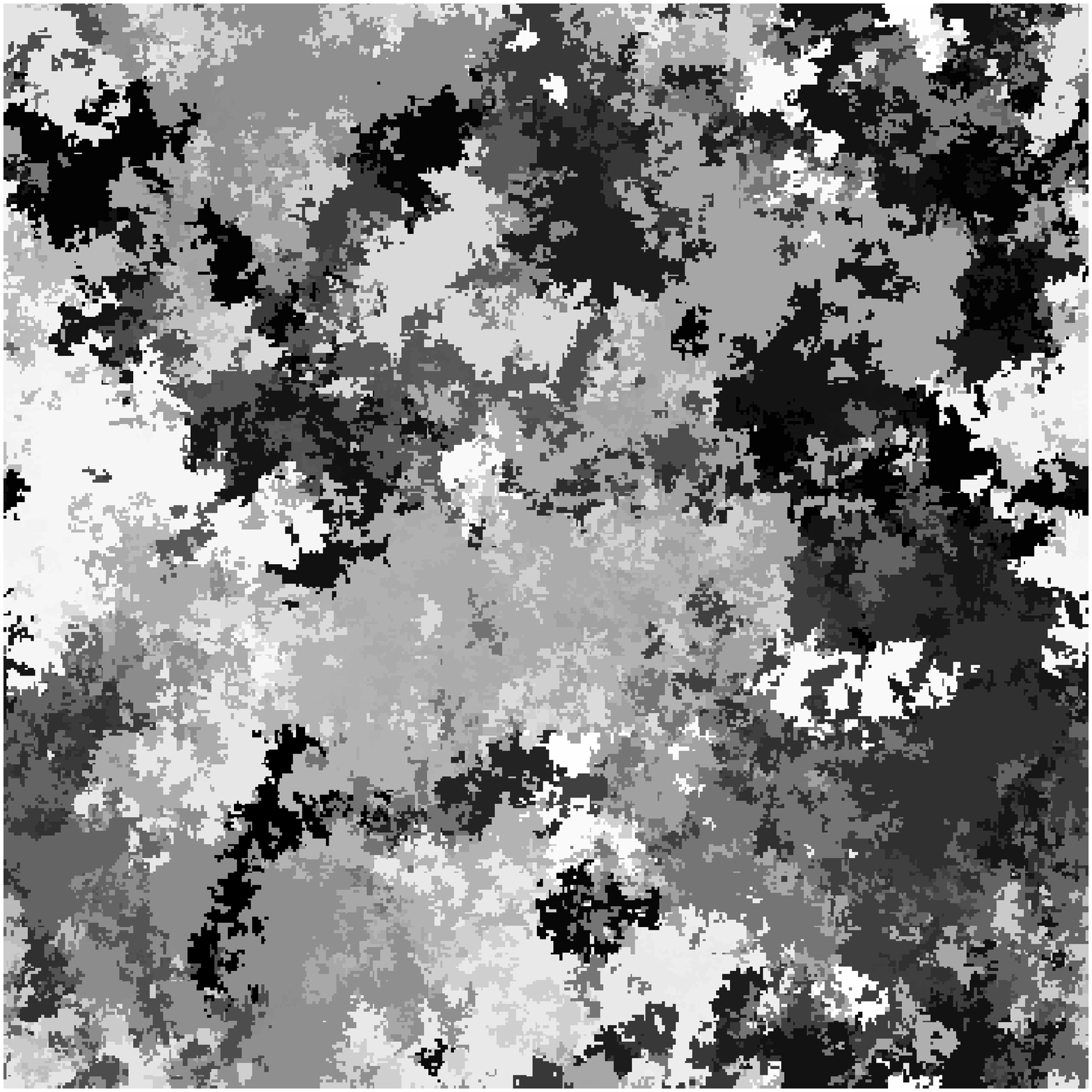, width = 150pt}} \hspace{5pt}
      \subfigure[$\ep = 1$ (55)]       {\epsfig{figure = 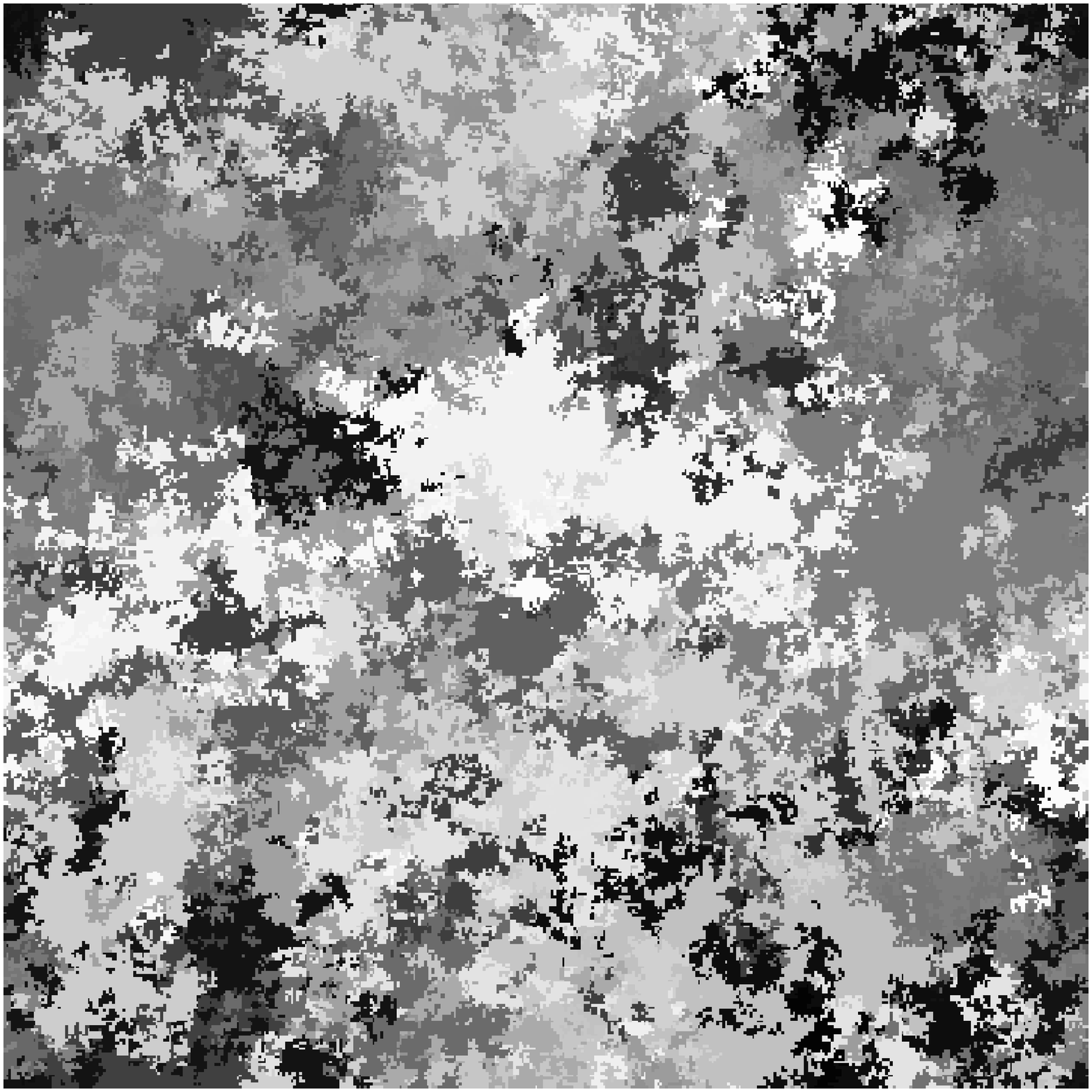, width = 150pt}}}
\caption{Snapshots of the process at time 10,000.}
\label{fig:thresholds}
\end{figure}

\indent While there is a lack of rigorous mathematical results for the Axelrod model, we prove analytically part of the
 results previously observed through numerical simulations as well as new results about the maximum number of opinions
 that can be supported by a graph for another variant of the voter model closely related to the Axelrod model which
 includes both social influence and a confidence threshold $\ep > 0$, i.e., pairs of actors whose opinion distance exceeds
 this threshold do not trust each other enough to mimic each other.
 The network of interactions is again described by a finite connected graph but each vertex is now characterized by an
 opinion taking values in $[0, 1]$.
 The opinion dynamics are similar to that of the voter model except that
 a vertex can mimic a neighbor only if the opinion distance between the vertex pair does not exceed the threshold $\ep$.
 The opinion dynamics are formally described by the continuous-time Markov process whose state at time $t$ is a function $\eta_t$
 that maps the vertex set into the opinion space $[0, 1]$, and whose Markov generator $L_V$ is defined on the set of cylinder
 functions by
 $$ L_V \,g (\eta) \ = \ \sum_{x \in V} \ \sum_{y \sim x} \ \ind \{|\eta (y) - \eta (x)| < \ep \} \ [g (\eta_{x \to y}) - g (\eta)] $$
 where $\eta_{x \to y}$ is obtained from $\eta$ by assigning the value $\eta (x)$ to vertex $y$.
 Note that, when the confidence threshold is set to $\ep = 0$, any configuration is an absorbing state regardless
 of the cardinality and topology of the network of interactions: the dynamics are frozen so that consensus are not possible.
 In contrast, when $\ep = 1$, one recovers the (multitype) voter model, in which case the duality with coalescing random walks
 implies that on any finite connected graph a total consensus is reached with probability 1.
 For results about consensus times of the voter model on the torus, we refer the reader to the article of \cite{cox_1989}.
 When $\ep \in (0, 1)$, the dynamics become more complex.
 Intuitively, agreements should emerge from the repeated interactions among agents, so one expects that the probability of a
 consensus is nondecreasing with respect to both the confidence threshold and the connectivity of the graph.
 This picture is supported by Figure \ref{fig:thresholds} that shows realizations of the process on a $400 \times 400$ lattice
 with periodic boundary conditions for different values of the confidence threshold $\ep$.
 Numbers between parenthesis at the bottom of each picture, which represent the number of opinions at time 10,000, strongly suggest
 that the expected number of opinions at equilibrium is nonincreasing with respect to $\ep$.
 It turns out that the probability of a consensus depends upon the initial configuration as well:
 we will prove that, whereas absorbing states with a large number of opinions can be artificially constructed in a deterministic
 fashion, starting from a random configuration in which the initial opinions are independent and uniformly distributed, the
 dynamics tend to greatly decrease the number of opinions present in the system, which is again supported by the numerical
 simulations of Figure \ref{fig:thresholds}.

\indent Before stating our results, recall that throughout this article $G = (V, E)$ denotes a finite connected graph with
 $N$ vertices.
 We employ the notation $\sim$ for the binary relation on the vertex set that indicates that two vertices are neighbors, i.e.,
 connected by an edge.
 Finally, given $\ep \in [0, 1]$, we let $P_{\ep}$ denote the law of the process with parameter $\ep$ starting from
 independent opinions chosen uniformly at random in the interval $[0, 1]$, and let $E_{\ep}$ be the corresponding expected value. \\


\indent{\bf Number of opinions at equilibrium -- static approach.}
 The first step is to look at the number of opinions that may coexist at equilibrium.
 Note that configurations in which all the agents share the same opinion, that we call consensus, are absorbing states, so
 the minimum number of opinions at equilibrium is trivially equal to 1.
 Let $\mu_{\ep} (G)$ denote the maximum number of opinions that may coexist at equilibrium, a quantity that we shall call
 the $\ep$-opinion index, or simply the opinion index, of the graph.
 A key tool to investigate the opinion index is to observe that the stationary distributions are exactly the ones supported
 on the set of absorbing states.
 This is intuitively obvious since the underlying graph is finite, but will be proved rigorously in the next section by applying
 the martingale convergence theorem.
 The problem thus reduces to finding all the \emph{static} configurations of the process, namely the
 configurations $\eta : V \longrightarrow [0, 1]$ such that
 $$ |\eta (x) - \eta (y)| > \ep \quad \hbox{or} \quad \eta (x) = \eta (y) \quad \hbox{for each edge} \ (x, y) \in E. $$
 The previous condition implies that, for any given graph $G$, the index $\mu_{\ep} (G)$ is nonincreasing with respect
 to the confidence threshold $\ep$.
 Moreover, we observe that the condition is always satisfied when the threshold $\ep = 0$, therefore $\mu_0 (G) = N$.
 In contrast, when $\ep = 1$, the process reduces to the multitype voter model that reaches eventually a consensus,
 therefore $\mu_1 (G) = 1$.

\indent For intermediate values of the confidence threshold, lower bounds of the opinion index can be expressed as a
 function of the chromatic number $\chi (G)$ of the graph.
 Recall that a graph is said to be $c$-colorable if one can color its vertices using at most~$c$ colors such that no two
 adjacent vertices receive the same color.
 The smallest number of colors needed to color a graph is called its chromatic number.
 Provided~$\ep < 1$, if a large graph can be colored using only few colors then one can construct absorbing states in which
 a large number of vertices (those colored with the same given color) have distinct opinions arbitrarily close to each other.
 In contrast, if the chromatic number is large, then one expects the number of opinions at equilibrium to be rather small.
 Precisely, by letting $\integer{a}$ be the least integer not less than $a$, we have the following theorem.

\begin{theorem}
\label{coloring}
 If $\ep < 1$ and $\chi = \chi (G)$ is the chromatic number of $G$ then
 $$ \begin{array}{rcl}
    \ep <    (\chi - 1)^{-1} & \Longrightarrow & \mu_{\ep} (G) \ = \ \card (V) \ = \ N. \vspace{4pt} \\
    \ep \geq (\chi - 1)^{-1} & \Longrightarrow & \mu_{\ep} (G) \ \geq \ \max (\integer{N / \chi} + 1, \integer{\ep^{-1}}). \end{array} $$
\end{theorem}

\noindent Theorem \ref{coloring} combined with results from the theory of graph coloring gives us insight into how the
 topology of the graph may affect its opinion index.
 It is not the purpose of this article to review all the results of graph coloring, so we only give few applications.

\begin{corollary}
 If $G$ is a planar graph and $\ep < 1/3$ then $\mu_{\ep} (G) = N$.
\end{corollary}
\begin{proof}
 This is a straightforward consequence of Theorem \ref{coloring} and the four color map theorem which states that any planar
 graph is 4-colorable.
\end{proof}

\begin{corollary}
 If $G$ is a triangle-free planar graph and $\ep < 1/2$ then $\mu_{\ep} (G) = N$.
\end{corollary}
\begin{proof}
 This follows from Theorem \ref{coloring} and Gr\"otzsch's theorem which states that every triangle-free planar graph is 3-colorable.
\end{proof}

\begin{corollary}
\label{bipartite}
 Assume that $1/2 < \ep < 1$.
 Then, $\mu_{\ep} (G) = N$ if and only if the graph $G$ is bipartite.
\end{corollary}
\begin{proof}
 If $G$ is bipartite then it is 2-colorable so $\mu_{\ep} (G) = N$ by Theorem \ref{coloring}.
 To prove the reverse, assume that the graph $G$ is not bipartite.
 Then it contains an odd cycle.
 Let $2n + 1$ be the length of this cycle.
 Without loss of generality, we may assume that at least $n + 1$ vertices in this cycle have opinion in $[0, 1/2]$,
 which induces the existence of two adjacent vertices with opinion in $[0, 1/2]$.
 Since $\ep > 1/2$, for the configuration to be an absorbing state, these two vertices must share the same opinion from which
 it follows that $\mu_{\ep} (G) < N$.
\end{proof}

\noindent In general, Theorem \ref{coloring} gives interesting results when considering graphs with small chromatic numbers,
 which gives rise to absorbing states with a large number of opinions.

\indent To understand the process on finite graphs with a large chromatic number such as complete graphs for
 which $\chi (G) = N$, we now give an upper bound of the opinion index $\mu_{\ep} (G)$.
 This upper bound is computed by considering a certain decreasing sequence of subgraphs in which each member is obtained
 from the previous one by removing its largest clique.
 Recall that a clique in a graph is a set of pairwise adjacent vertices inducing a subgraph which is a complete graph.
 Let $G$ be a finite graph, and $W \subset V$ its largest clique.
 We introduce the following notations:
\begin{enumerate}
 \item[--] The number of vertices in $W$ is denoted by $\omega (G)$. \vspace{4pt}
 \item[--] The subgraph of $G$ induced by the vertex set $V \setminus W$ is denoted by $\pi (G)$.
\end{enumerate}
 When the graph has more than one maximal clique, the subgraph $\pi (G)$ is not uniquely defined, and the upper bound
 in the next theorem is obtained by taking the minimum over all the possible choices of the sequence of subgraphs.

\begin{theorem}
\label{clique}
 Assume that $G = (V, E)$ is a finite connected graph. Then,
 $$ \begin{array}{rcl}
    \mu_{\ep} (G) & \leq & \displaystyle
    \min_{\hbox{\tiny choices}} \,(\min (\omega (G), \integer{\ep^{-1}}) \ + \ \mu_{\ep} (\pi (G))) \vspace{4pt} \\ & \leq &
    \displaystyle \min_{\hbox{\tiny choices}} \,(\min (\omega (G), \integer{\ep^{-1}}) \ + \ \cdots \ + \
    \min (\omega (\pi^{N - 1} (G)), \integer{\ep^{-1}})). \end{array} $$
\end{theorem}

\noindent The sum is stopped at 
$N - 1$ because 
$\pi^i (G) = \varnothing$  
whenever $i \geq N$.
 As previously explained, while Theorem \ref{coloring} gives interesting lower bounds for graphs with small chromatic number,
 Theorem~\ref{clique} gives interesting upper bounds for graphs with large connectivity.
 In particular, the result turns out to be optimal in the case of the complete graph for
 which $\mu_{\ep} (G) = \min (N, \integer{\ep^{-1}})$. \\


\indent{\bf Number of opinions at equilibrium -- dynamic approach.}
 We now investigate the limiting number of opinions $\nu_{\ep} (G)$ that coexist at equilibrium in a \emph{dynamic}
 context, namely for the process starting from the configuration in which the initial opinions $\eta_0 (x)$, $x \in V$,
 are independent and uniformly distributed in the interval $[0, 1]$.
 It is important to point out that $\nu_{\ep} (G)$ is a random variable while the opinion index $\mu_{\ep} (G)$ is a
 deterministic upper bound.

\indent For any finite connected graph, one has $\nu_{\ep} (G) = \mu_{\ep} (G)$ when the confidence threshold is either equal
 to $\ep = 0$ or equal to $\ep = 1$.
 The opinion index, however, might be rather crude for intermediate values of the confidence
 threshold $\ep$.
 This is due to the fact that stable configurations with $\mu_{\ep} (G)$ opinions are often artificially constructed in a
 deterministic way and have in fact a small probability.
 For instance, the graph used to produce our simulation pictures is an example of bipartite graph therefore
 $\mu_{\ep} (G) = 400^2$ for all $\ep < 1$ according to Corollary \ref{bipartite}.
 In contrast, our simulation results indicate that for reasonably large values of the confidence threshold, most of the
 opinions are lost.
 In support to this picture, our next result gives a general lower bound for the probability of a consensus that holds
 regardless of the topology of the graph.

\begin{theorem}
\label{consensus}
 Assume that $\ep > 1/2$. Then,
 $$ P_{\ep} \,(\hbox{consensus}) \ = \ P_{\ep} \,(\nu_{\ep} (G) = 1) \ \geq \ 2 \ep - 1. $$
\end{theorem}

\noindent Theorem \ref{consensus} follows from an application of the optional sampling theorem that shows
 that the set of agents whose opinion lies outside $(1 - \ep, \ep)$, that can be seen as ``extremist'' agents,
 hits the empty set with probability $2 \ep - 1$.
 The proof uses the fact that all ``centrist'' agents, with opinion in the interval $(1 - \ep, \ep)$, may interact
 with all ``extremist'' agents.
 Once all agents are ``centrist'' the system evolves according to a voter model so the population reaches a consensus.

\indent We now look at the long-term behavior of the process when $\ep > 0$ is small.
 Observe first that, when the degree of the graph is uniformly bounded by a positive constant, say $K$, the number of edges
 in the graph is at most $NK$, hence
 $$ \begin{array}{rcl}
     P_{\ep} \,(\nu_{\ep} (G) \neq N) & \leq &
     P_{\ep} \,(|\eta_0 (x) - \eta_0 (y)| < \ep \ \hbox{for some} \ (x, y) \in E) \vspace{8pt} \\ & \leq &
    \displaystyle \sum_{x \sim y} \ P_{\ep} \,(|\eta_0 (x) - \eta_0 (y)| < \ep) \ \leq \ 2 \ep \,\card (E) \ \leq \
     2 \ep \,N K. \end{array} $$
 In particular, with probability close to 1, the number of opinions coexisting at equilibrium is equal to the number
 $N$ of vertices provided $\ep = o (1/N)$.
 Note that the upper bound above depends on the number of vertices, which is necessary for the number of opinions
 coexisting at equilibrium to be exactly equal to $N$, or equivalently, for the initial state to be an absorbing state.
 This is due to the fact that, when $\ep > 0$ does not depend on $N$, with probability close to 1 when $N$ is large, the
 number of opinions at equilibrium will be strictly less than $N$ since one can find with high probability at least one
 edge whose endpoints are initially within opinion distance $\ep$.
 Such edges, however, are minoriy, so we conjecture that, with probability close to 1 when $N$ is large, any given
 fraction $c < 1$ of the opinions can ultimately be retained by the dynamics provided $\ep$ lies below a positive constant
 that only depends upon $c$ and the local topology of the graph.
\begin{conjecture}
\label{conjecture}
 For any $c < 1$, there exists $\ep_0 (c, K) \in (0, 1)$ such that, uniformly over all graphs with degree at most $K$
 and number of vertices $N$, we have
 $$ P_{\ep} \,(\nu_{\ep} (G) < c N) \ \leq \ C \exp (- \gamma N) \quad \hbox{for all} \ \ep < \ep_0 (c, K) $$
 and suitable constants $C > 0$ and $\gamma \in (0, \infty)$ that only depends on $K$ and $\ep$.
\end{conjecture}
\noindent Using an analogy with bond percolation, any configuration of opinions can be thought of as inducing a partition
 of the vertex set into clusters, where two vertices $x$ and $y$ belong to the same cluster if and only if there is
 a graph path in $G$ from $x$ to $y$ with $\ep$-agreement along it.
 Peierls argument implies that, when $\ep$ is sufficiently small, the size of a typical cluster at time 0 decays
 exponentially, therefore the initial number of clusters can be made larger than $cN$ with high probability.
 Then, under the assumption that the opinion distance between any two vertices belonging to adjacent clusters
 initially exceeds $\ep$, all the clusters are closed under the dynamics: an opinion can at most spreads within its
 parent cluster.
 In particular, the ultimate number of coexisting opinions dominates the initial number of clusters, which
 suggests the conjecture.
 This argument breaks down if the initial configuration contains at least one cluster that is not closed under the
 dynamics, which occurs with probability close to 1 due to the presence of arbitrarily large clusters at time 0 that
 contain a wide range of opinions, however the analogy with bond percolation supported by numerical simulations gives
 us the following insight into the geometry of the ultimate configuration when the confidence threshold is small:
 most opinions that persist have only few representatives, those that originate from clusters that are initially closed
 under the dynamics, and few opinions have a significantly larger number of representatives, those that originate from
 clusters that are not closed under the dynamics.
 Based on a fairly different approach far from percolation theory, the conjecture can be established for path graphs,
 i.e., two of the vertices have degree 1 and all others degree 2.
\begin{theorem}
\label{coexistence}
 Assume that $G = (V, E)$ is a path. Then
 $$ P_{\ep} \,(\nu_{\ep} (G) < (1 - c_1 \,\ep) N) \ \leq \ \exp (- c_2 \,\ep N) $$
 for some constants $c_1, c_2 > 0$ that do not depend on $N$ and $\ep$.
\end{theorem}
\noindent The conjecture when the graph is a path follows by taking $\ep = (1 - c) / c_1$.
 In order to prove this result, we will look at the process coupled with the opinion dynamics that keeps track of the
 disagreements along the edges rather than the actual opinion at each vertex.
 It will be proved that the number of opinions on a path graph, and more generally the number of connected components
 in which all vertices share the same opinion on a tree, is nearly equal to the number of edges along which a disagreement
 persists, therefore the objective will be to bound the ultimate number of edges whose endpoints agree.
 When the graph is a path, as opposed to a tree with larger degree or a graph including loops, the dynamics on
 the edges can be expressed in a simple manner, which is the key to finding a suitable upper bound.
 This will be done by ultimately ignoring the spatial structure and exhibiting a connection with a simple urn problem.


\section{Proof of Theorems \ref{coloring} and \ref{clique}}
\label{sec:static}

\indent To study the opinion dynamics and characterize their equilibriums, it is convenient to construct the process
 graphically using an idea of \cite{harris_1972}.
 We first orient each edge $e \in E$ arbitrarily.
 Each edge is further equipped with a Poisson process with parameter 2 whose $n$th arrival time is denoted by $T_n (e)$
 together with a collection of independent coin flips by letting
 $$ P \,(U_n (e) = + 1) \ = \ P \,(U_n (e) = - 1) \ = \ 1/2 \quad \hbox{for all} \ n \geq 1. $$
 Poisson processes and collections of coin flips attached to different edges are independent.
 The process starting from any initial configuration is constructed from this graphical representation in the following way.
 Let $e = (x, y)$ being oriented from $x$ to $y$, and let $a$ and $b$ denote the opinion at $x$ and $y$ at time $T_n (e)$.
\begin{enumerate}
 \item If $|a - b| > \ep$ then nothing happens at time $T_n (e)$. \vspace{4pt}
 \item If $|a - b| \leq \ep$ and $U_n (e) = + 1$ then the opinion of $y$ jumps from $b$ to $a$. \vspace{4pt}
 \item If $|a - b| \leq \ep$ and $U_n (e) = - 1$ then the opinion of $x$ jumps from $a$ to $b$.
\end{enumerate}
 Any absorbing state of the process induces a stationary distribution.
 Since the vertex set is finite, it should be clear that any stationary distribution is in turn supported
 on the set of absorbing states, which we now prove rigorously using the martingale convergence theorem.

\begin{lemma}
\label{equilibrium}
 Assume that $G = (V, E)$ is finite.
 Then, any stationary distribution is supported on the set of absorbing states.
\end{lemma}
\begin{proof}
 For $x \in V$, let $\Theta_t (x)$ be the set of vertices with opinion $\eta_0 (x)$ at time $t$, that is the set of vertices
 whose opinion at time $t$ originates from $x$ at time 0, i.e.,
 $$ \Theta_t (x) = \{y \in V : \eta_t (y) = \eta_0 (x) \}. $$
 The set-valued process $\Theta_t (x)$ may only evolve at times $T_n (e)$ when the oriented edge $e$ connects a vertex
 $y \in \Theta_t (x)$ to a vertex $z \in V \setminus \Theta_t (x)$.
 Without loss of generality, set $e = (y, z)$ is oriented from vertex $y$ to vertex $z$.
\begin{enumerate}
 \item If $|\eta_t (y) - \eta_t (z)| > \ep$ then nothing happens at time $T_n (e)$. \vspace{4pt}
 \item If $|\eta_t (y) - \eta_t (z)| \leq \ep$ then $\theta_t (x) = \card (\Theta_t (x))$ jumps to $\theta_t (x) + U_n (e)$ at the
  arrival time $T_n (e)$.
\end{enumerate}
 Since the random variable $U_n (e)$ takes values $+ 1$ and $- 1$ with equal probabilities, this implies that the number of vertices
 with opinion $\eta_0 (x)$ is a martingale.
 In particular, the martingale convergence theorem implies the existence of a random variable $\theta_{\infty} (x) \in \{0, 1, \ldots, N \}$ such that
 $$ \lim_{t \to \infty} \ \theta_t (x) = \theta_{\infty} (x) \quad \hbox{with probability 1}. $$
 Since the state space of $\theta_t (x)$ is finite, there exist a stopping time $T_x$ almost surely finite and a random subset
 of $V$ that we denote by $\Theta_{\infty} (x)$ such that
 $$ \theta_t (x) = \theta_{\infty} (x) \quad \hbox{and} \quad \Theta_t (x) = \Theta_{\infty} (x) \quad \hbox{for all} \ t \geq T_x. $$
 Finally, set $T = \max \,\{T_x : x \in V \}$.
 Then, $T$ is almost surely finite and
 $$ \Theta_t (x) = \Theta_{\infty} (x) \quad \hbox{for all $x \in V$ and $t \geq T$} $$
 indicating that, starting from any initial configuration, the process eventually gets trapped into an absorbing state.
 This completes the proof.
\end{proof}

\begin{lemma}
\label{subgraph}
 Assume that $G_1 = (V, E_1)$ and $G_2 = (V, E_2)$ with $E_1 \subset E_2$.
 Then, we have $\mu_{\ep} (G_1) \geq \mu_{\ep} (G_2)$.
\end{lemma}
\begin{proof}
 By Lemma \ref{equilibrium}, there exists a configuration $\eta$ with $\mu_{\ep} (G_2)$ distinct opinions which is an absorbing
 state of the process evolving on $G_2$. In particular, we have
 $$ |\eta (x) - \eta (y)| > \ep \quad \hbox{or} \quad \eta (x) = \eta (y) \qquad \hbox{for each edge} \ (x, y) \in E_2. $$
 This property is true for the edges of $G_1$ since $E_1 \subset E_2$, which implies that $\eta$ is also an absorbing state for the
 process evolving on $G_1$.
 In conclusion, we have the inequality $\mu_{\ep} (G_1) \geq \mu_{\ep} (G_2)$.
\end{proof}

\begin{lemma}
\label{partition}
 Assume that $G = (V, E)$ is a finite connected graph.
 Let $\{V_1, V_2 \}$ be a partition of $V$, and let $G_i = (V_i, E_i)$ be the subgraph induced by $V_i$. Then,
 $$ \mu_{\ep} (G) \ \leq \ \mu_{\ep} (G_1) + \mu_{\ep} (G_2). $$
\end{lemma}
\begin{proof}
 Since $E_i \subset E$, any absorbing state $\eta : V \longrightarrow [0, 1]$ for the process evolving on the graph $G$ induces
 an absorbing state $\eta_i$ for the process on $G_i$.
 In particular, there exist two absorbing states for the processes evolving on $G_1$ and $G_2$, respectively,
 such that the sum of the number of distinct opinions in each configuration is larger than $\mu_{\ep} (G)$.
 The result follows.
\end{proof}

\begin{lemma}
\label{complete}
 If $G = (V, E)$ is a complete graph then $\mu_{\ep} (G) = \min (N, \integer{\ep^{-1}})$.
\end{lemma}
\begin{proof}
 Let $J = \integer{\ep^{-1}}$ and let $\eta : V \longrightarrow [0, 1]$ be a configuration with at least $J + 1$ opinions.
 Then, there exist vertices $x_0, x_1, \ldots, x_J$ such that
 $$ 0 \ < \ \eta (x_0) \ < \ \eta (x_1) \ < \ \cdots \ < \ \eta (x_J) \ < \ 1. $$
 In particular, we can find $j \in \{0, 1, \ldots, J - 1 \}$ such that
 $$ |\eta (x_{j + 1}) - \eta (x_j)| \ \leq \ J^{-1} \ = \ 1 / \integer{\ep^{-1}} \ \leq \ \ep. $$
 Since vertices $x_j$ and $x_{j + 1}$ are neighbors (complete graph), the previous inequality implies that the
 configuration $\eta$ is not an absorbing state.
 It then follows from Lemma \ref{equilibrium} that $J + 1$ opinions cannot coexist at equilibrium so
 $$ \mu_{\ep} (G) \ \leq \ J \ = \ \integer{\ep^{-1}}. $$
 To conclude, it suffices to construct an absorbing state with $\min (N, J)$ distinct opinions.
 Assume first that $\min (N, J) = J$ and set
 $$ \eta (x_j) = \min (j / (J - 1), 1) \quad \hbox{for} \ j = 0, 1, \ldots, N - 1 $$
 where $x_0, x_1, \ldots, x_{N - 1}$ denote the vertices.
 Since $\ep < (J - 1)^{-1}$ the configuration thus defined is an absorbing state with $J$ distinct opinions.
 Finally, when $\min (N, J) = N$ we have $\ep < (N - 1)^{-1}$ and the same argument implies the existence of an
 absorbing state with $N$ opinions.
\end{proof} \\ \\
 With Lemmas \ref{equilibrium}-\ref{complete} in hands, we can now prove Theorems \ref{coloring} and \ref{clique}. \\


\begin{figure}[t]
\begin{center}
 \epsfig{figure = 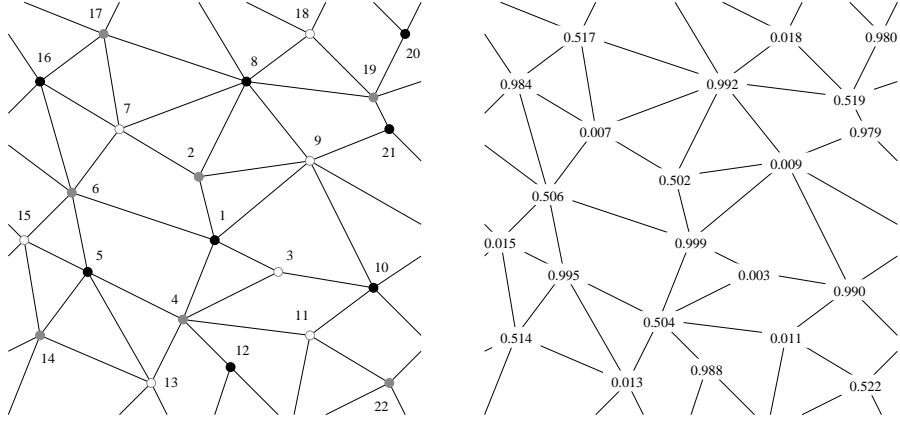, width = 340pt}
\end{center}
 \caption{\upshape{$\chi = 3$ and $\alpha = 0.001$}}
\label{fig:color}
\end{figure}

\begin{demon}{coloring}
 Assume first that $\ep < (\chi - 1)^{-1}$ and fix a coloring of the vertices using exactly $\chi$ colors, say $1, 2, \ldots, \chi$,
 such that no two adjacent vertices receive the same color.
 Also introduce a total order on the vertex set by setting
 $$ V \ = \ \{x_1, x_2, \ldots, x_N \} $$
 and fix $\alpha > 0$ sufficiently small such that $\ep + 2 N \alpha < (\chi - 1)^{-1}$.
 If vertex $x_i$ has received color $j \neq \chi$ then we set
 $$ \eta (x_i) \ = \ (j - 1) \cdot (\chi - 1)^{-1} + i \alpha $$
 whereas if vertex $x_i$ has received color $j = \chi$ then we set
 $$ \eta (x_i) \ = \ (j - 1) \cdot (\chi - 1)^{-1} - i \alpha \ = \ 1 - i \alpha. $$
 See Figure \ref{fig:color} for an example of construction on a 3-colorable graph.
 This induces a configuration with $N$ distinct opinions.
 Moreover, if vertices $x_i$ and $x_j$ have received different colors then
 $$ |\eta (x_i) - \eta (x_j)| \ > \ (\chi - 1)^{-1} - 2 N \alpha \ > \ \ep. $$
 Since no two adjacent vertices receive the same color, the configuration $\eta$ is an absorbing state, which establishes the first
 part of Theorem \ref{coloring}.
 We now assume that $\ep \geq (\chi - 1)^{-1}$, and consider the same coloring as before.
 Let $N_j$ denote the number of vertices with color $j$. Since
 $$ N_1 + N_2 + \cdots + N_{\chi} \ = \ N \quad \hbox{and} \quad N_j \in \Z_+ \quad \hbox{for} \ j = 1, 2, \ldots, \chi $$
 there is $j_0$ such that $N_{j_0} \geq \integer{N / \chi}$.
 Let $\alpha > 0$ such that $\ep + N \alpha < 1$, and define
 $$ \eta (x_i) \ = \ \left\{\begin{array}{cl} i \alpha & \hbox{for all $x_i \in V$ with color $j_0$} \vspace{2pt} \\
                                              1 & \hbox{for all $x_i \in V$ with color $j \neq j_0$} \end{array} \right. $$
 Since no two adjacent vertices have color $j_0$ and $1 - N \alpha > \ep$, this defines an absorbing state with
 $N_{j_0} + 1 \geq \integer{N / \chi} + 1$ opinions.
 To prove that $\mu_{\ep} (G) \geq \integer{\ep^{-1}}$, we consider the complete graph $\bar G$ with the same vertex set $V$ as $G$,
 and apply Lemmas \ref{subgraph} and \ref{complete} to get
 $$ \mu_{\ep} (G) \ \geq \ \mu_{\ep} (\bar G) \ = \ \min (N, \integer{\ep^{-1}}) \ = \ \integer{\ep^{-1}}. $$
 The last equality follows from the fact that $\integer{\ep^{-1}} \leq \chi - 1 < N$.
\end{demon} \\


\begin{demon}{clique}
 Let $V_1 = W$ and $V_2 = V \setminus W$, where $W \subset V$ is a maximal clique in the graph $G$.
 Then, the subgraph $G_1$ induced by $V_1$ is the complete graph with $\omega (G)$ vertices, and the subgraph induced by $V_2$
 is $G_2 = \pi (G)$ therefore, combining Lemmas \ref{partition} and \ref{complete}, we obtain
 $$ \mu_{\ep} (G) \ \leq \ \mu_{\ep} (G_1) \ + \ \mu_{\ep} (G_2) \ \leq \ \min (\omega (G), \integer{\ep^{-1}}) \ + \ \mu_{\ep} (\pi (G)). $$
 Since this holds for any choice of $W$, the result follows by induction.
\end{demon}


\section{Proof of Theorem \ref{consensus}}
\label{sec:consensus}

\noindent Recall that $\ep > 1/2$. We introduce
 $$ \Theta_t = \{x \in V : \eta_t (x) \notin (1 - \ep, \ep) \} \quad \hbox{and} \quad \theta_t = \card (\Theta_t). $$
 Using the notations of the proof of Lemma \ref{equilibrium}, we have
 $$ \Theta_t \ = \ \bigcup_{x \in \Theta_0} \Theta_t (x) \quad \hbox{and} \quad \theta_t \ = \ \sum_{x \in \Theta_0} \ \theta_t (x). $$
 The proof of Lemma \ref{equilibrium} implies that $\theta_t (x)$ is a martingale so the process $\theta_t$ also is a martingale.
 Therefore, by the martingale convergence theorem, there exists a random variable $\theta_{\infty}$ such that
 $$ \lim_{t \to \infty} \ \theta_t = \theta_{\infty} \quad \hbox{with probability 1}. $$
 If $\theta_t \notin \{0, N \}$ then there exists $(x, y) \in E$ such that $x \in \Theta_t$ and $y \notin \Theta_t$.
 Since any opinion in the interval $(1 - \ep, \ep)$ is within distance $\ep$ of both endpoints 0 and 1, we can conclude that
 vertices $x$ and $y$ have the following three properties:
 $$ |\eta_t (x) - \eta_t (y)| \leq \ep \qquad \eta_t (x) \neq \eta_t (y) \qquad \hbox{and} \qquad x \sim y, $$
 which implies that $\eta_t$ is not an absorbing state.
 Since by Lemma \ref{equilibrium} the process converges to an absorbing state, we deduce that the random variable $\theta_{\infty}$
 takes values in $\{0, N \}$.
 Let $T$ be the first time the process hits an absorbing state.
 The optional sampling theorem implies that
 $$ \begin{array}{rcl}
    E_{\ep} \,(\theta_T) & = &
    E_{\ep} \,(\theta_0) \ = \ N \times P_{\ep} \,(x \in \Theta_0) \ = \ 2 (1 - \ep) N \vspace{4pt} \\ & = &
    E_{\ep} \,(\theta_{\infty}) \ = \ 0 \times P_{\ep} \,(\theta_{\infty} = 0) + N \times P_{\ep} \,(\theta_{\infty} = N), \end{array} $$
 from which it follows that $P_{\ep} \,(\theta_{\infty} = 0) = 1 - 2 (1 - \ep) = 2 \ep - 1$.
 To conclude, we observe that on the event that $\theta_{\infty} = 0$, all the opinions present in the system after time $T$ are
 within distance $2 \ep - 1 \leq \ep$ of each other so the process evolves according to a voter model.
 In particular,
 $$ P_{\ep} \,(\nu_{\ep} (G) = 1) \ = \ P_{\ep} \,(\hbox{consensus}) \ \geq \ P_{\ep} \,(\theta_{\infty} = 0) \ = \ 2 \ep - 1. $$
 This completes the proof of Theorem \ref{consensus}.


\section{Proof of Theorem \ref{coexistence}}
\label{sec:coexistence}

\indent The key to proving Theorem \ref{coexistence} is to investigate a Markov process coupled with the opinion dynamics
 that keeps track of the disagreements along the edges of the graph rather than the actual opinion at each vertex.
 The state at time $t$ can be viewed as a weighted graph where each edge is assigned a weight that measures the opinion
 distance between its endpoints.
 The process evolves until each edge has either weight zero or an absolute weight larger than the confidence threshold $\ep$,
 which corresponds to an absorbing state.
 The strategy is to first exhibit a relationship between the ultimate number of coexisting opinions and the ultimate
 number of edges with weight zero, and then prove that the probability that the fraction of edges with weight zero exceeds
 a certain threshold decreases exponentially with the number of vertices. \\


\indent{\bf Dynamics on a general finite connected graph.}
 To define dynamics on the edges of the graph coupled with the opinion dynamics, we first equip the vertex set with
 a total order relationship by setting $V = \{x_1, x_2, \ldots, x_N \}$, which implicitly induces an orientation of
 the edges.
 That is, we think of each edge of the graph as being oriented with
 $$ \hbox{$e = (x_i, x_j) \in E$ being oriented from $x_i$ to $x_j$ if and only if $i < j$}. $$
 By convention, if $e = (x, y)$ is oriented from $x$ to $y$, we write $e_1 = x$ and $e_2 = y$.
 Then, the opinion dynamics on the vertex set described by the process $\eta_t$ along with an arbitrary orientation of
 the graph naturally induce dynamics on the edges by letting
 $$ \xi_t (e) \ = \ \eta_t (e_2) - \eta_t (e_1) \quad \hbox{for all} \ e \in E. $$
 The stochastic process thus defined is a continuous-time Markov process whose state space includes all mappings of
 the edge set $E$ into the interval $[-1, 1]$.
 Moreover, since a vertex adopts the opinion of a given neighbor at rate 1 if and only if the opinion distance between
 both vertices is less than $\ep$, the dynamics on the edge set are described by the Markov generator $L_E$ defined by
 $$ L_E \,g (\xi) \ = \ \sum_{e \in E} \ \ind \{|\xi (e)| < \ep \} \ [g (\xi_{+e}) + g (\xi_{-e}) - 2 \,g (\xi)] $$
 where $\xi_{+e}$ and $\xi_{-e}$ are the configurations on the edges given by
 $$ \begin{array}{rcl}
    \xi_{+e} (e') & = & \left\{\hspace{-3pt} \begin{array}{cl}
                               \xi (e') + \xi (e) & \hbox{when} \ \ e'_1 = e_2 \vspace{3pt} \\
                               \xi (e') - \xi (e) & \hbox{when} \ \ e'_2 = e_2 \end{array} \right. \vspace{4pt} \\
    \xi_{-e} (e') & = & \left\{\hspace{-3pt} \begin{array}{cl}
                               \xi (e') - \xi (e) & \hbox{when} \ \ e'_1 = e_1 \vspace{3pt} \\
                               \xi (e') + \xi (e) & \hbox{when} \ \ e'_2 = e_1 \end{array} \right. \end{array} $$
 while $\xi_{+e}$ and $\xi_{-e}$ coincide with $\xi$ otherwise.
 Note that $\xi_{+e} (e) = \xi_{-e} (e) = 0$ due to an agreement between the endpoints of $e$.
 Thinking of $\xi (e)$ as a weight, either a positive or a negative weight, of edge $e$, the evolution rules can be
 described informally as follows.
 Independently and at rate two, the edge $e = (x, y)$ chooses one of its endpoints, say $x$, uniformly at random.
 If the absolute value of the weight of edge $e$ is less than $\ep$, then the weight of each of the edges incident to
 $x$ is updated by summing or subtracting, depending on the orientation of the graph, the weight of $e$, which
 results in particular in the weight of $e$ becoming zero.
 Otherwise, nothing happens.
 Note that if originally edge $e$ has weight zero then the previous rule has no effect on the configuration.
 In particular, a configuration such that each edge has either weight zero or an absolute weight larger than $\ep$ is
 an absorbing state, and vice versa.
 This motivates the following definition.
\begin{definition}
\label{type}
 For $j = 1, 2, \ldots, J = \integer{\ep^{-1}}$, an edge $e \in E$ is said to be
\begin{enumerate}
 \item empty at time $t$ whenever $\xi_t (e) = 0$, \vspace{4pt}
 \item of type $j$ at time $t$ if its absolute weight is such that $(j - 1) \,\ep < |\xi_t (e)| < j \ep$.
\end{enumerate}
\end{definition}
\noindent In the previous definition, weights equal to a multiple of $\ep$ different from 0 are ignored because, as the difference
 between two numbers taken from a set of independent continuous random variables, they appear with probability 0. \\


\indent{\bf Dynamics on a path-like graph.}
 When the graph is a path, the vertex set is equipped with a natural order relationship such that
 $$ x_i \sim x_j \quad \hbox{if and only if} \quad |i - j| = 1. $$
 For simplicity, we shall use this order relationship to study the process on a path as it implies that two
 edges interacting result in one empty edge and one edge whose weight is obtained by summing, rather than subtracting,
 the weights of the edges:
 $$ \xi_{+e} (e') \ = \ \xi_{-e} (e') \ = \ \xi (e) + \xi (e') \quad
    \hbox{for all} \ e, e' \in E \ \hbox{with} \ e \cap e' \neq \varnothing. $$
 To describe the dynamics, it is also convenient to add two vertices, say $x_0$ and $x_{N + 1}$, which are connected by a
 virtual edge to $x_1$ and $x_N$, respectively.
 Then, the process on a path can be simply described as follows: each weight on a non-virtual edge with absolute value
 less than $\ep$ is displaced at rate 2 to one of the adjacent edges chosen uniformly at random.
 Note that the weight of both virtual edges is unimportant since it has no effect on the weight of the other edges. \\


\indent{\bf Number of opinions on trees and path-like graphs.}
 As previously mentioned, the number of opinions and the number of empty edges are closely related.
 More precisely, we prove that for a path-like graph the sum of these two quantities is always equal to the number of vertices.
 We first prove one inequality which holds for general connected graphs with no cycles.
 The result is sharp in the sense that it fails for connected graphs with cycles.

\begin{lemma}
\label{tree}
 Assume that the connected graph $G = (V, E)$ is a tree. Then
 $$ \card \,\{\eta_t (x) : x \in V \} \ + \ \card \,\{e \in E : \xi_t (e) = 0 \} \ \leq \ N. $$
\end{lemma}
\begin{proof}
 As a subgraph of a tree, the set of empty edges at time $t$ along with all $N$ vertices of the graph is a forest.
 Since two vertices belonging to the same tree in this forest are connected by a path of empty edges, they share the same opinion.
 In particular, the total number of opinions is at most equal to the number of trees in the forest, say $T$.
 Using in addition that for any tree the number of vertices minus the number of edges equals one, we also have that $N$ is equal
 to the number of empty edges plus the number of trees $T$.
 In summary, we have
 $$ \card \,\{\eta_t (x) : x \in V \} \ \leq \ T \quad \hbox{and} \quad \card \,\{e \in E : \xi_t (e) = 0 \} \ = \ N - T $$
 from which the lemma follows.
\end{proof} \\ \\
 Having a graph with a self-avoiding loop of length $m$ and assuming that all the vertices in this loop share the same opinion
 but that all the other vertices have different opinions gives
 $$ \card \,\{\eta_t (x) : x \in V \} \ = \ N - m + 1 \quad \hbox{and} \quad \card \,\{e \in E : \xi_t (e) = 0 \} \ = \ m, $$
 so the assumption that the connected graph is a tree in the previous lemma is necessary.
 The next result states that the inequality in Lemma \ref{tree} becomes an equality on a path-like graph, which is again sharp in
 the sense that this does not hold if at least one vertex has degree 3.
\begin{lemma}
\label{path}
 If $G = (V, E)$ is a path and initial opinions are all distinct then
 $$ \card \,\{\eta_t (x) : x \in V \} \ + \ \card \,\{e \in E : \xi_t (e) = 0 \} \ = \ N. $$
\end{lemma}
\begin{proof}
 Since the initial opinions are all distinct,
 $$ \card \,\{\eta_0 (x) : x \in V \} \ = \ N \quad \hbox{and} \quad \card \,\{e \in E : \xi_0 (e) = 0 \} \ = \ 0. $$
 Hence, the property to be proved is true at time 0.
 To prove that it holds at any time, we introduce the following two definitions.
 We say that an opinion is lost at time $t$ if the last representative of an opinion mimics one of its neighbors at time $t$,
 which is equivalent to
 $$ \card \,\{\eta_t (x) : x \in V \} \ = \ \card \,\{\eta_{t-} (x) : x \in V \} \ - \ 1. $$
 Similarly, we say that a weight is lost at time $t$ if an update occurs at vertex $x$ at time $t$ when all the edges
 incident to $x$ have weight different from 0, which is equivalent to
 $$ \card \,\{e \in E : \xi_t (e) = 0 \} \ \geq \ \card \,\{e \in E : \xi_{t-} (e) = 0 \} \ + \ 1. $$
 Now, assume that an opinion is lost at time $t$.
 Then, there exists $x \in V$ such that
 $$ \hbox{(i)}  \ \ \eta_{t-} (x) \neq \eta_{t-} (z) \quad \hbox{for all} \ z \neq x \quad \hbox{and} \quad
    \hbox{(ii)} \ \ \eta_t (x) = \eta_t (y) \quad \hbox{for some} \ y \sim x. $$
 Applying property (i) above for all $z \sim x$ and letting $e = (x, y)$, it follows that
\begin{enumerate}
 \item $\xi_{t-} (e') \neq 0$ for all edges $e' \in E$ with one endpoint at $x$ and \vspace{4pt}
 \item $\xi_t (e) = \eta_t (y) - \eta_t (x) = 0$,
\end{enumerate}
 which implies that a weight is lost at time $t$.
 Also, we observe that at each update of the process one edge becomes empty.
 Because on a path graph the weight of at most two edges can change simultaneously, it follows that the number of empty edges
 cannot decrease.
 In conclusion, since each time an opinion is lost a weight is also lost, and the number of empty edges cannot decrease,
 the number of opinions lost up to time $t$ is smaller than the number of weights lost up to time $t$ which, together with
 the fact that the initial number of opinions is $N$, implies that
 $$ \card \,\{\eta_t (x) : x \in V \} \ + \ \card \,\{e \in E : \xi_t (e) = 0 \} \ \geq \ N. $$
 The reverse inequality follows from Lemma \ref{tree}.
 This completes the proof.
\end{proof} \\ \\
\noindent Note that the assumption that the graph is a path-like graph is necessary as shown by the simple counter-example
 of Figure \ref{fig:trees}.
 In this example, the opinion dynamics start with five independent uniformly distributed opinions satisfying the condition in the
 caption.
 Dashed lines refer to empty edges, continuous lines to edges of type 1, and dotted lines to edges with absolute weight more than $\ep$.
 The final configuration on the right-hand side is an absorbing state with only two empty edges and two distinct opinions while the
 graph has five vertices.
 Note also that the previous lemma indicates that, on a path, an opinion is lost at time $t$ if and only if exactly one weight is
 lost at time $t$, therefore when the weight of an edge of type 1 is added to the weight of another edge, the resulting weight is
 almost surely different from 0.
 More generally,
 $$ \xi_t (e) \ + \ \xi_t (e') \ = \ 0 \quad \hbox{implies that} \quad \xi_t (e) \ = \ \xi_t (e') \ = \ 0 \quad \hbox{for all} \ e, e' \in E. $$
 This result can be proved directly without invoking Lemma \ref{path} by using the one-dimensionality of the graph and the fact that
 each weight is the difference of two values chosen from a finite set of independent continuous random variables. \\

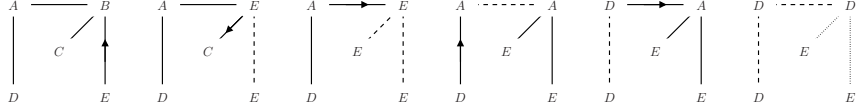
\begin{figure}[t]
\centering
\scalebox{0.24}{\input{trees.pstex_t}}
\caption{\upshape{Realization on a tree with $|E - A|$, $|E - B|$, $|E - C|$, $|D - A| < \ep < |E - D|$.}}
\label{fig:trees}
\end{figure}


\indent{\bf Connection with a simple urn problem.}
 We assume from now on that the graph is a path-like graph, and focus on the proof of our last
 result, Theorem \ref{coexistence}.
 To deduce the theorem from Lemma \ref{path}, it suffices to find a suitable upper bound for the ultimate number of empty edges.
 The idea is to study the maximum number of weights that are lost before absorption over all the possible realizations of the
 process by ignoring the spatial structure, i.e., at each update, an edge of type 1 along with another edge are chosen uniformly
 at random to interact. Let
 $$ X_t (j) \ = \ \card \,\{e \in E : (j - 1) \,\ep < |\xi_t (e)| < j \ep \} \quad \hbox{for} \ j = 1, 2, \ldots, J = \integer{\ep^{-1}} $$
 denote the number of edges of type $j$, and $X_t (0)$ the number empty edges at time $t$.
 The first step is to bound the initial number of edges of any type, as shown in the following lemma.
\begin{lemma}
\label{binomial}
 Let $J = \integer{\ep^{-1}}$ and $M = \integer{4 \ep N} - 1$. Then,
 $$ P_{\ep} \,(X_0 (j) > M) \ \leq \ \exp (- \ep N) \quad \hbox{for all} \ j = 1, 2, \ldots, J. $$
\end{lemma}
\begin{proof}
 We first fix $j$ and observe that, since the random variables $\eta_0 (x)$, $x \in V$, are uniformly distributed in the
 interval $[0, 1]$, we have
 $$ \begin{array}{l}
     P_{\ep} \,((j - 1) \,\ep < |\xi_0 (e)| < j \ep \ | \ \eta_0 (e_1) = a) \ = \
     P_{\ep} \,((j - 1) \,\ep < |\eta_0 (e_2) - a| < j \ep) \vspace{4pt} \\ \hspace{40pt} \leq \
     P_{\ep} \,(a + (j - 1) \,\ep < \eta_0 (e_2) < a + j \ep) \ + \ P_{\ep} \,(a - j \ep < \eta_0 (e_2) \vspace{4pt} \\ \hspace{40pt} < \
     a - (j - 1) \,\ep) \ \leq \ 2 \ep \end{array} $$
 for all $e = (e_1, e_2) \in E$ and $a \in [0, 1]$.
 This, together with the fact that the graph does not contain any loops and that the initial opinions are independent,
 implies that the initial number of edges of type $j$ is stochastically smaller than the binomial random variable
 $X \sim \bin \,(N - 1, 2 \ep)$.
 Note that the first parameter corresponds to the number of edges of the graph which, in
 the case of a path and more generally a tree, is equal to the number of vertices minus one.
 In particular, it follows from large deviation estimates for the Binomial distribution that
 $$ \begin{array}{rcl}
     P_{\ep} \,(X_0 (j) > M) & \leq & P \,(X > \integer{4 \ep N} - 1) \vspace{4pt} \\ & \leq &
        \exp (- 2 N \ep^2 / (2 \ep)) \ \leq \ \exp (- \ep N). \end{array} $$
 This completes the proof.
\end{proof} \\

\begin{figure}[t]
\centering
\includegraphics[width = 340pt]{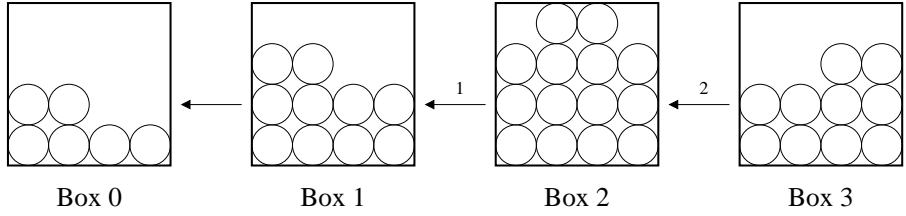}
\caption{\upshape{Schematic illustration of strategy $S$}}
\label{fig:boxes}
\end{figure}

\noindent Now, we observe that an interaction between an edge of type 1 and an empty edge does not affect the number of edges of
 each type.
 Therefore, since we now ignore the structure of the network, we shall also ignore such events, i.e., each realization of the
 process corresponds to a random sequence of edges of type 1 becoming empty interacting with non-empty edges.
 In particular, using that each time a weight is lost, the number of empty edges increases by 1 according to Lemma \ref{path},
 the ultimate number of empty edges is also equal to the number of updates before absorption of the process.
 Also observe that an interaction between an edge of type 1 and an edge of type $j$ results in the first edge becoming empty and
 the second edge becoming of type either $j - 1$ or $j$ or $j + 1$.

\indent To keep track of the process $X_t$ we now play the following game:
 we consider a set of $J + 1$ boxes, labelled from box 0 to box $J$, with box $j \neq 0$ containing initially a certain random number
 of balls and box 0 being initially empty.
 Then, at each time step, we do the following.
\begin{enumerate}
 \item Edge becoming empty: move a ball from box 1 to box 0, \vspace{3pt}
 \item Second interacting edge: choose a box $j \neq 0$ uniformly at random from the set of all the non-empty boxes, \vspace{3pt}
 \item Updated weight: move a ball from this box to one of the three boxes $j - 1$, $j$ or $j + 1$ chosen
  uniformly at random with probability $1/3$.
\end{enumerate}
 The game stops when box 1 is empty.
 We now think of the number of balls in box $j \neq 0$ as the number of edges of type $j$, and the number of balls in box 0 as the
 number of empty edges, and observe that standard coupling arguments imply that the number of steps to empty box 1 increases stochastically
 if we add more balls at time 0.
 In particular, the reasoning above implies that, given that the initial number of edges of each type is at most $M = \integer{4 \ep N} - 1$,
 the ultimate number of empty edges is bounded by the maximum number of steps over all the possible realizations of the game to
 empty box 1 when starting with $M$ balls in each box $j \neq 0$, which is reached for the strategy that will keep box 1 non-empty as long
 as possible.
 Since in addition at least $j - 1$ steps are required to move a ball from box $j$ to box 1 (exactly $j - 1$ steps by always moving the
 ball to the left), the maximum number of steps to empty box 1 over all the possible realizations of the game is reached when following the
 deterministic strategy $S$ that consists at each step in
\begin{enumerate}
 \item Moving a ball from box 1 to box 0, \vspace{3pt}
 \item Choosing the non-empty box with the lowest label $j \geq 2$, \vspace{3pt}
 \item Taking a ball from this box and moving it to box $j - 1$.
\end{enumerate}
 Letting $Y_n (j)$ denote the deterministic number of balls in box $j$ at step $n$ under strategy $S$ when starting with exactly $M$
 balls in each box $j \neq 0$, we obtain the following values:
 $$ Y_n (1) = M \quad \hbox{and} \quad Y_n (2) = M - n $$
 for all $n = 0, 1, \ldots, M$, and
 $$ Y_{n + M} (1) = M - \integer{n / 2}, \quad Y_{n + M} (2) = \ind \{n \ \hbox{odd} \} \quad \hbox{and} \quad Y_{n + M} (3) = M - \integer{n / 2} $$
 for all $n = 1, \ldots, 2 M$.
 Balls from box $j \geq 4$ are never moved and the game halts after $3 M$ steps, which implies that the event that
 $$ \lim_{t \to \infty} \ X_t (0) > 3 M \quad \hbox{and} \quad X_0 (j) \leq M \ \ \hbox{for} \ j = 1, 2, 3 $$
 is the empty event.
 This and Lemmas \ref{path} and \ref{binomial} allow to conclude that
 $$ \begin{array}{l}
     P_{\ep} \,(\nu_{\ep} (G) < (1 - 12 \ep) N) \ \leq \
     P_{\ep} \,(\nu_{\ep} (G) < N - 3 M) \ \leq \
     P_{\ep} \,(\lim_{t \to \infty} X_t (0) > 3 M) \vspace{4pt} \\ \hspace{20pt} \leq \
     P_{\ep} \,(\lim_{t \to \infty} X_t (0) > 3 M \ \hbox{and} \ X_0 (j) \leq M \ \hbox{for} \ j = 1, 2, 3) \vspace{4pt} \\ \hspace{60pt} + \
     P_{\ep} \,(\lim_{t \to \infty} X_t (0) > 3 M \ \hbox{and} \ X_0 (j) > M \ \hbox{for some} \ j = 1, 2, 3) \vspace{4pt} \\ \hspace{20pt} \leq \
     P_{\ep} \,(\varnothing) \ + \
     P_{\ep} \,(\lim_{t \to \infty} X_t (0) > 3 M \ \hbox{and} \ X_0 (j) > M \ \hbox{for some} \ j = 1, 2, 3) \vspace{4pt} \\ \hspace{20pt} \leq \
     0 \ + \ 3 \ \exp (- \ep N). \end{array} $$
 This completes the proof of Theorem \ref{coexistence}. \\


\noindent\textbf{Acknowledgment}.
 The author would like to thank an anonymous referee for pointing out several mistakes in a preliminary version.


\bibliographystyle{alea2}
\bibliography{07-01}

\end{document}

%% file: trees.pstex_t
\begin{picture}(0,0)%
\includegraphics{trees.pstex}%
\end{picture}%
\setlength{\unitlength}{3947sp}%
\begingroup\makeatletter\ifx\SetFigFontNFSS\undefined%
\gdef\SetFigFontNFSS#1#2#3#4#5{%
  \reset@font\fontsize{#1}{#2pt}%
  \fontfamily{#3}\fontseries{#4}\fontshape{#5}%
  \selectfont}%
\fi\endgroup%
\begin{picture}(22824,3024)(-461,-1873)
\put(  1,689){\makebox(0,0)[b]{\smash{{\SetFigFontNFSS{20}{24.0}{\familydefault}{\mddefault}{\updefault}$A$}}}}
\put(2401,689){\makebox(0,0)[b]{\smash{{\SetFigFontNFSS{20}{24.0}{\familydefault}{\mddefault}{\updefault}$B$}}}}
\put(1201,-511){\makebox(0,0)[b]{\smash{{\SetFigFontNFSS{20}{24.0}{\familydefault}{\mddefault}{\updefault}$C$}}}}
\put(  1,-1711){\makebox(0,0)[b]{\smash{{\SetFigFontNFSS{20}{24.0}{\familydefault}{\mddefault}{\updefault}$D$}}}}
\put(2401,-1711){\makebox(0,0)[b]{\smash{{\SetFigFontNFSS{20}{24.0}{\familydefault}{\mddefault}{\updefault}$E$}}}}
\put(3901,689){\makebox(0,0)[b]{\smash{{\SetFigFontNFSS{20}{24.0}{\familydefault}{\mddefault}{\updefault}$A$}}}}
\put(5101,-511){\makebox(0,0)[b]{\smash{{\SetFigFontNFSS{20}{24.0}{\familydefault}{\mddefault}{\updefault}$C$}}}}
\put(3901,-1711){\makebox(0,0)[b]{\smash{{\SetFigFontNFSS{20}{24.0}{\familydefault}{\mddefault}{\updefault}$D$}}}}
\put(6301,-1711){\makebox(0,0)[b]{\smash{{\SetFigFontNFSS{20}{24.0}{\familydefault}{\mddefault}{\updefault}$E$}}}}
\put(6301,689){\makebox(0,0)[b]{\smash{{\SetFigFontNFSS{20}{24.0}{\familydefault}{\mddefault}{\updefault}$E$}}}}
\put(7801,689){\makebox(0,0)[b]{\smash{{\SetFigFontNFSS{20}{24.0}{\familydefault}{\mddefault}{\updefault}$A$}}}}
\put(7801,-1711){\makebox(0,0)[b]{\smash{{\SetFigFontNFSS{20}{24.0}{\familydefault}{\mddefault}{\updefault}$D$}}}}
\put(10201,-1711){\makebox(0,0)[b]{\smash{{\SetFigFontNFSS{20}{24.0}{\familydefault}{\mddefault}{\updefault}$E$}}}}
\put(10201,689){\makebox(0,0)[b]{\smash{{\SetFigFontNFSS{20}{24.0}{\familydefault}{\mddefault}{\updefault}$E$}}}}
\put(9001,-511){\makebox(0,0)[b]{\smash{{\SetFigFontNFSS{20}{24.0}{\familydefault}{\mddefault}{\updefault}$E$}}}}
\put(11701,689){\makebox(0,0)[b]{\smash{{\SetFigFontNFSS{20}{24.0}{\familydefault}{\mddefault}{\updefault}$A$}}}}
\put(11701,-1711){\makebox(0,0)[b]{\smash{{\SetFigFontNFSS{20}{24.0}{\familydefault}{\mddefault}{\updefault}$D$}}}}
\put(14101,-1711){\makebox(0,0)[b]{\smash{{\SetFigFontNFSS{20}{24.0}{\familydefault}{\mddefault}{\updefault}$E$}}}}
\put(12901,-511){\makebox(0,0)[b]{\smash{{\SetFigFontNFSS{20}{24.0}{\familydefault}{\mddefault}{\updefault}$E$}}}}
\put(14101,689){\makebox(0,0)[b]{\smash{{\SetFigFontNFSS{20}{24.0}{\familydefault}{\mddefault}{\updefault}$A$}}}}
\put(21901,-1711){\makebox(0,0)[b]{\smash{{\SetFigFontNFSS{20}{24.0}{\familydefault}{\mddefault}{\updefault}$E$}}}}
\put(20701,-511){\makebox(0,0)[b]{\smash{{\SetFigFontNFSS{20}{24.0}{\familydefault}{\mddefault}{\updefault}$E$}}}}
\put(19501,689){\makebox(0,0)[b]{\smash{{\SetFigFontNFSS{20}{24.0}{\familydefault}{\mddefault}{\updefault}$D$}}}}
\put(19501,-1711){\makebox(0,0)[b]{\smash{{\SetFigFontNFSS{20}{24.0}{\familydefault}{\mddefault}{\updefault}$D$}}}}
\put(21901,689){\makebox(0,0)[b]{\smash{{\SetFigFontNFSS{20}{24.0}{\familydefault}{\mddefault}{\updefault}$D$}}}}
\put(15601,-1711){\makebox(0,0)[b]{\smash{{\SetFigFontNFSS{20}{24.0}{\familydefault}{\mddefault}{\updefault}$D$}}}}
\put(18001,-1711){\makebox(0,0)[b]{\smash{{\SetFigFontNFSS{20}{24.0}{\familydefault}{\mddefault}{\updefault}$E$}}}}
\put(16801,-511){\makebox(0,0)[b]{\smash{{\SetFigFontNFSS{20}{24.0}{\familydefault}{\mddefault}{\updefault}$E$}}}}
\put(18001,689){\makebox(0,0)[b]{\smash{{\SetFigFontNFSS{20}{24.0}{\familydefault}{\mddefault}{\updefault}$A$}}}}
\put(15601,689){\makebox(0,0)[b]{\smash{{\SetFigFontNFSS{20}{24.0}{\familydefault}{\mddefault}{\updefault}$D$}}}}
\end{picture}%

%% file: 07-01.bbl
\begin{thebibliography}{11}
\providecommand{\natexlab}[1]{#1}
\providecommand{\url}[1]{\texttt{#1}}
\providecommand{\urlprefix}{URL }
\expandafter\ifx\csname urlstyle\endcsname\relax
  \providecommand{\doi}[1]{doi:\discretionary{}{}{}#1}\else
  \providecommand{\doi}{doi:\discretionary{}{}{}\begingroup
  \urlstyle{rm}\Url}\fi
\providecommand{\eprint}[2][]{\url{#2}}

\bibitem[{Axelrod(1997)}]{axelrod_1997}
R.~Axelrod.
\newblock The dissemination of culture: a model with local convergence and
  global polarization.
\newblock \emph{J. Conflict Resolut.} \textbf{41}, 203--226 (1997).

\bibitem[{Castellano et~al.(2009)Castellano, Fortunato and
  Loreto}]{castellano_fortunato_loreto_2009}
C.~Castellano, S.~Fortunato and V.~Loreto.
\newblock Statistical physics of social dynamics.
\newblock \emph{Reviews of Modern Physics} \textbf{81}, 591--646 (2009).

\bibitem[{Castellano et~al.(2000)Castellano, Marsili and
  Vespignani}]{castellano_all_2000}
C.~Castellano, M.~Marsili and A.~Vespignani.
\newblock Nonequilibrium phase transition in a model for social influence.
\newblock \emph{Phys. Rev. Lett.} \textbf{85}, 3536--3539 (2000).

\bibitem[{Clifford and Sudbury(1973)}]{clifford_sudbury_1973}
P.~Clifford and A.~Sudbury.
\newblock A model for spatial conflict.
\newblock \emph{Biometrika} \textbf{60}, 581--588 (1973).
\newblock \href{http://www.ams.org/mathscinet-getitem?mr=343950}{MR343950}.

\bibitem[{Cox(1989)}]{cox_1989}
J.~T. Cox.
\newblock Coalescing random walks and voter model consensus times on the torus
  in {${\mathbb Z}\sp d$}.
\newblock \emph{Ann. Probab.} \textbf{17}, 1333--1366. (1989).
\newblock \href{http://www.ams.org/mathscinet-getitem?mr=MR1048930}{MR1048930}.

\bibitem[{Gonz\'alez-Avella et~al.(2005)Gonz\'alez-Avella, Cosenza and
  Tucci}]{gonzalez_all_2005}
J.~C. Gonz\'alez-Avella, M.~G. Cosenza and K.~Tucci.
\newblock Nonequilibrium transition induced by mass media in a model for social
  influence.
\newblock \emph{Phys. Rev. E} \textbf{72}, 065102 (2005).
\newblock \href{http://www.ams.org/mathscinet-getitem?mr=MR1048930}{MR1048930}.

\bibitem[{Harris(1972)}]{harris_1972}
T.~E. Harris.
\newblock Nearest neighbor markov interaction processes on multidimensional
  lattices.
\newblock \emph{Adv. Math.} \textbf{9}, 66--89 (1972).
\newblock \href{http://www.ams.org/mathscinet-getitem?mr=MR0307392}{MR0307392}.

\bibitem[{Holley and Liggett(1975)}]{holley_liggett_1975}
R.~A. Holley and T.~M. Liggett.
\newblock Ergodic theorems for weakly interacting systems and the voter model.
\newblock \emph{Ann. Probab.} \textbf{3}, 643--663 (1975).
\newblock \href{http://www.ams.org/mathscinet-getitem?mr=MR0402985}{MR0402985}.

\bibitem[{Klemm et~al.(2003)Klemm, Egu\`iluz, Toral and
  San~Miguel}]{klemm_all_2003}
K.~Klemm, V.~M. Egu\`iluz, R.~Toral and M.~San~Miguel.
\newblock Role of dimensionality in axelrod's model for the dissemination of
  culture.
\newblock \emph{Physica A} \textbf{327}, 1--5 (2003).
\newblock \href{http://www.ams.org/mathscinet-getitem?mr=MR2027091}{MR2027091}.

\bibitem[{Vazquez and Redner(2007)}]{vazquez_redner_2007}
F.~Vazquez and S.~Redner.
\newblock Non-monotonicity and divergent time scale in axelrod model dynamics.
\newblock \emph{EPL} \textbf{78}, 18002 (2007).

\bibitem[{Vilone et~al.(2002)Vilone, Vespignani and
  Castellano}]{vilone_all_2002}
D.~Vilone, A.~Vespignani and C.~Castellano.
\newblock Ordering phase transition in the one-dimensional axelrod model.
\newblock \emph{Eur. Phys. J. B} \textbf{30}, 399--406 (2002).

\end{thebibliography}
